\documentclass[12pt,a4paper]{article}
\usepackage{latexsym,amssymb,amsfonts,amsmath,amsthm,nccmath,cases}
\usepackage{amsmath,amsthm,fullpage,mdwlist}
\usepackage{pdfsync}
\usepackage{amssymb, amsfonts, graphicx}
\usepackage[usenames,dvipsnames]{xcolor}
\usepackage{enumitem}
\usepackage[margin=2cm]{geometry}
\usepackage{multicol}

\numberwithin{equation}{section}
\newtheorem{theorem}{Theorem}[section]
\newtheorem{corollary}[theorem]{Corollary}
\newtheorem{lemma}[theorem]{Lemma}

\theoremstyle{definition}

\newtheorem{conjecture}[theorem]{Conjecture}
\newtheorem*{remark}{Remark}

\def\P{\mathcal{P}}
\def\C{\mathcal{C}}
\def\HH{\mathcal{H}_{38}}

\def\eref#1{$(\ref{#1})$}
\def\sref#1{\S$\ref{#1}$}
\def\lref#1{Lemma~$\ref{#1}$}
\def\tref#1{Theorem~$\ref{#1}$}
\def\Tref#1{Table~$\ref{#1}$}
\def\cjref#1{Conjecture~$\ref{#1}$}
\def\cyref#1{Corollary~$\ref{#1}$}

\def \deg {{\rm deg}}

\def \leq {\leqslant}
\def \geq {\geqslant}
\def \le {\leqslant}
\def \ge {\geqslant}

\def \mod#1{{\:({\rm mod}\ #1)}}

\let\oldproofname=\proofname
\renewcommand{\proofname}{\rm\bf{\oldproofname}}

\def\eps{\varepsilon}

\newcommand{\ignore}[1]{}

\title{On Ryser's Conjecture for Linear\\
Intersecting Multipartite Hypergraphs\footnote{Research
supported by ARC grant DP120100197}}

\author{Nevena Franceti\'{c}\footnote{
School of Mathematical Sciences, Monash University Victoria 3800 Australia.} 
\quad Sarada Herke \footnotemark[2] \quad Brendan D. McKay\footnote{Research School of 
Computer Science, Australian National University,
Canberra, ACT 0200, Australia}\quad Ian M. Wanless\footnotemark[2]}
\date{}

\begin{document}

\maketitle

\begin{abstract}
Ryser conjectured that $\tau\le(r-1)\nu$ for $r$-partite hypergraphs,
where $\tau$ is the covering number and $\nu$ is the matching number.
We prove this conjecture for $r\le9$ in the special case of \emph{linear 
intersecting}
hypergraphs, in other words where every pair of lines meets in exactly
one vertex.

Aharoni formulated a stronger version of Ryser's conjecture which
specified that each $r$-partite hypergraph should have a cover of size
$(r-1)\nu$ \emph{of a particular form}.  We provide a counterexample to
Aharoni's conjecture with $r=13$ and $\nu=1$.

We also report a number of computational results.  For $r=7$, we find
that there is no linear intersecting hypergraph that achieves the
equality $\tau=r-1$ in Ryser's conjecture, although non-linear
examples are known.  We exhibit intersecting non-linear examples
achieving equality for $r\in\{9,13,17\}$. Also, we find that $r=8$ is
the smallest value of $r$ for which there exists a linear intersecting
$r$-partite hypergraph that achieves $\tau=r-1$ and is not isomorphic to a
subhypergraph of a projective plane.
\end{abstract}

\section{Introduction}


A \emph{hypergraph} $H$ is a set of non-empty subsets, variously
called \emph{lines, edges} or \emph{hyperedges}, of a finite
underlying vertex set $V(H)$.  The \emph{degree} of a vertex $v\in
V(H)$, denoted $\deg(v)$, is the number of lines in $H$ that contain
$v$. A hypergraph is \emph{$r$-uniform} if every line contains exactly
$r$ vertices. Thus a $2$-uniform hypergraph is simply a graph. 

Covers and matchings in hypergraphs are widely studied \cite{F88}.  A
\emph{cover} of a hypergraph $H$ is a set of vertices ${C}\subseteq
V(H)$ such that every line of $H$ contains at least one vertex of
${C}$.  The \emph{covering number} of $H$, denoted $\tau(H)$, is the
minimum size of a cover of $H$.  A \emph{matching} in $H$ is a set of
pairwise disjoint lines of $H$ and the \emph{matching number} of $H$,
denoted $\nu(H)$, is the maximum size of a matching in $H$. Most
hypergraphs in this paper are \emph{intersecting}, meaning that every
pair of lines meets in at least one vertex; equivalently $\nu=1$.

The covering number and matching number of a hypergraph are related.
First, for every hypergraph, $\nu \leq \tau$ since each cover contains
at least one vertex from each line in any given matching.  Second, for
every $r$-uniform hypergraph, $\tau \leq r\nu$ since a cover can be
obtained from the union of the lines in a maximal matching; this bound
is sharp and is achieved, for example, by projective planes of order $r-1$.

A hypergraph is \emph{$r$-partite} if its vertex set can be
partitioned into $r$ sets, called \emph{sides}, such that every line
consists of exactly one vertex from each side.  Hence every
$r$-partite hypergraph is necessarily $r$-uniform.  An $r$-partite
hypergraph can be constructed from a projective plane of order $r-1$
by removing a single vertex $v$ and the $r$ lines through~$v$.  The
resulting hypergraph $\P'_r$ is called a \emph{truncated projective
  plane} and the $r$ sides of $\P'_r$ are defined by the sets of
vertices on each of the removed lines.  Our notation hides the fact
that there may be non-isomorphic choices for the structure of $\P'_r$,
but the distinction between the different choices will not matter in
our work.


The following famous conjecture is due to Ryser~\cite{R67}.

\begin{conjecture}\label{conj:Ryser}
Every $r$-partite hypergraph with covering number $\tau$ and matching
number $\nu$ satisfies $\tau \leq (r-1)\nu$.
\end{conjecture}

Ryser's Conjecture is far from being resolved.  
When $r=2$, the conjecture is K\"{o}nig's Theorem for bipartite 
graphs, which is also equivalent to Hall's Theorem~(see e.g.~\cite{vLW01}).  
Using topological methods, Aharoni and Haxell \cite{AH00} proved a hypergraph 
generalisation of Hall's Theorem and Aharoni \cite{A01} used this result to 
prove \cjref{conj:Ryser} for $r=3$. 
\cjref{conj:Ryser} was proved for intersecting $r$-partite hypergraphs 
with $r\leq5$ by Tuza \cite{T83}.  Haxell and Scott \cite{HS12} 
built on this result to prove that for $r\le5$, 
there exists $\epsilon>0$ such that $\tau < (r-\epsilon)\nu$
for all $r$-partite hypergraphs.

In this paper, we prove another restricted version of Ryser's
Conjecture for small $r$. A hypergraph is \emph{linear} if each pair
of lines meets in at most one vertex.  Hence, in a \emph{linear
  intersecting} hypergraph, each pair of lines intersects in exactly
one vertex.  In~\sref{sec:linear_upto_9} we prove that
\cjref{conj:Ryser} holds for $r \leq 9$
when $H$ is a linear intersecting hypergraph.
The proof of the $r=9$ case is computational.


In an intersecting $r$-partite hypergraph each side is a cover and each
line is also a cover. In~\cite{ABW15} the following strenthening
of \cjref{conj:Ryser} is recorded. The conjecture is due to Aharoni,
as are several generalisations that are also given in~\cite{ABW15}.

\begin{conjecture}\label{conj:variation}
An intersecting $r$-partite hypergraph $H$ has a side of size at most
$r-1$ or a cover of the form $e\setminus\{v\}$ for some $e \in H$ and
$v \in e$.
\end{conjecture}

In~\sref{sec:counterex} we show that \cjref{conj:variation} is false,
by providing a counterexample when $r = 13$.
Furthermore, in \sref{sec:mols} we describe an infinite family of
linear intersecting $r$-partite hypergraphs with $\tau=r-2$, built
from mutually orthogonal latin squares.  Although this family of
hypergraphs have covers consisting of an line with a vertex removed,
they have the property that no minimal cover is contained within a
line or a side.

In studying \cjref{conj:Ryser} it is natural to investigate hypergraphs
that achieve the equality $\tau=(r-1)\nu$.  A well-known infinite
family of such hypergraphs with $\nu=1$ is the family of truncated projective
planes $\P'_r$ (where $r-1$ is a prime power).  Note that $\tau \leq
r-1$ since each side of $\P'_r$ has $r-1$ vertices and $\tau \ge r-1$
since each vertex lies on $r-1$ lines and the total number of lines is
$(r-1)^2$.  However, $\P'_r$ has more lines than is necessary in the sense
that many of its subhypergraphs achieve $\tau=r-1$. Conversely, we will
report in \sref{sec:comp} that for $r\le7$ the only way to achieve
$\tau=r-1$ in a linear intersecting $r$-partite hypergraph is to take
a subhypergraph of $\P'_r$.  In particular, there are no
linear intersecting $7$-partite hypergraphs with
$\tau=6$, since $\P'_7$ does not exist. By contrast, it was shown in
\cite{AP,ABW15} that there are non-linear intersecting $7$-partite hypergraphs
with $\tau=6$.  We also describe in~\sref{s:H38} a linear
intersecting hypergraph having $r=8$ and $\tau=7$ and which is not a
subhypergraph of $\P_8'$. Moreover, in~\sref{sec:non_prime_power} we
give examples of non-linear intersecting $r$-partite hypergraphs with
$r \in \{9,13,17\}$ which have covering number $r-1$.




\subsection{Notation}

We deal with $r$-partite hypergraphs throughout. To avoid degeneracies
we always assume that $r\ge2$ and that every vertex has positive
degree. The sides of our hypergraphs are always denoted
$V_0,V_1,\dots,V_{r-1}$ and we have $V=\cup_{i=0}^{r-1}V_i$.  
The covering number, matching number, minimum degree and maximum degree of
a hypergraph are denoted by $\tau$, $\nu$, $\delta$ and $\Delta$,
respectively. The number of lines in a hypergraph $H$ is denoted $h$
or $|H|$. We often use discrete interval notation $[n_1,n_2] = \{n_1,
n_1+1, n_1+2, \dots, n_2 \}$, where $n_1$ and $n_2$ are two integers
and $n_1 \leq n_2$.

\section{Ryser's conjecture for linear intersecting hypergraphs}\label{sec:linear_upto_9} 

In this section we prove that \cjref{conj:Ryser} holds for all linear
intersecting hypergraphs with at most nine sides.  We begin by
establishing some properties of a hypothetical counterexample.

\subsection{General properties}

\begin{lemma}\label{l:tau_r_requirements}
Let $H$ be an intersecting $r$-partite hypergraph with $\tau(H) = r$.
Then
\begin{itemize}\itemsep=0pt
\item[(i)] Each side of $H$ has size at least $r$.
\item[(ii)] Each vertex of $H$ has degree at least $2$.
\item[(iii)] Each line of $H$ contains at most one vertex of degree $2$.
\end{itemize}
\end{lemma}

\begin{proof}
Each side is a cover so it contains at least $\tau(H)$ vertices,
hence (i) holds.  If a vertex $v$ has degree $1$ and $\ell$ is the line
containing $v$, then $\ell \backslash \{v\}$ is a cover of size $r-1$,
hence (ii) holds.  Finally, if there is a line $\ell$ containing distinct
vertices $u$ and $v$ of degree $2$, then we get a cover of size $r-1$
by taking $\ell \backslash \{u,v\} \cup \{x\}$, where $x$ is a vertex
in the intersection of the (at most) two lines other than $\ell$ that meet 
$\{u,v\}$.  Therefore, (iii) is proved.
\end{proof}

\begin{lemma}\label{l:maxdegalmdisj}
Let $H$ be an intersecting $r$-partite hypergraph with
$\tau(H)=r$. Then $\Delta \ge 4.$ Furthermore, if $H$ is linear, then 
$\Delta\le r-2$.
\end{lemma}

\begin{proof}
  By \lref{l:tau_r_requirements}, we may assume that $\delta(H)\ge2$.
  Suppose there are $h$ lines in $H$. If we count the lines which
  intersect a given line, we find that $h \leq(\Delta-1)r+1$ with
  equality only if $H$ is $\Delta$-regular.  If we count the lines
  incident with a side that contains a vertex of degree $\Delta$, we
  find that $h \ge 2(r-1)+\Delta$, with equality only if all other
  vertices on that side have degree~$2$.  The above observations
  together show that $\Delta \ge 4$.

  Now suppose that $H$ is linear and let $v$ be a vertex of degree
  $\Delta$. Without loss of generality, assume that $v$ is on side
  $V_0$ and let $u\in V_0\setminus\{v\}$ (such a $u$ exists, since
  $\tau(H)>1$).  Any line $\ell$ through $u$ meets the lines through
  $v$ in $\Delta$ distinct vertices on sides other than $V_0$, since
  $H$ is linear and intersecting. Hence $\Delta\leq r-1$.  If
  $\Delta=r-1$ then side $V_1$ has at most $r-1$ vertices.  Indeed,
  $\ell$ cannot contain any vertex on side $V_1$ other than the
  vertices on lines through $v$. As this is true for any line $\ell$
  which does not contain $v$, there can only be $r-1$ vertices in
  $V_1$, which means that $\tau(H)\leq r-1$. We conclude that
  $\Delta\leq r-2$.
\end{proof}

\begin{theorem}\label{t:quadratic_arg}
Let $H$ be an intersecting $r$-partite hypergraph
with $\tau(H)=r$ and maximum degree $\Delta$. Then 
\begin{equation}\label{e:lem}
(r+1/2)^2 \Delta^2+4r^2\ge(8r^2-2r)\Delta.
\end{equation}
\end{theorem}

\begin{proof}
Suppose that $H$ contains $d_i$ vertices of degree $i$ for $2\leq
i\leq\Delta$.  Since each side of $H$ is a cover, there must be
$r^2+\eps$ vertices in $H$ for some $\eps\ge0$. Hence
\begin{equation}\label{e:di}
r^2+\eps=\sum_{i=2}^\Delta d_i.
\end{equation}
By counting vertex-line incidences we find that 
\begin{equation}\label{e:idi}
rh=\sum_{i=2}^\Delta id_i.
\end{equation}
Also, since $H$ is intersecting, we get the following inequality 
by counting incidences between lines:
\begin{equation}\label{e:iidi}
{h \choose 2}\leq \sum_{i=2}^\Delta{i\choose2}d_i.
\end{equation}
Our aim is to show that (\ref{e:iidi}) cannot be satisfied unless
(\ref{e:lem}) holds.  Since, by \lref{l:tau_r_requirements}(iii), no line of 
$H$ may contain
more than one vertex of degree $2$, 
$d_2=h/2-\eps'$ for some $\eps'\ge0$. Now solving
(\ref{e:di}) and (\ref{e:idi}) for $d_3$ and $d_\Delta$ 
(recall that $\Delta>3$ by \lref{l:maxdegalmdisj}) and substituting
into (\ref{e:iidi}), we get 
\begin{align}
h^2-(\Delta r+&\Delta/2+2r)h+3r^2\Delta \label{e:quadl}\\
&\leq 3\Delta\sum_{i=4}^{\Delta-1}d_i
-(\Delta+2)\sum_{i=4}^{\Delta-1}id_i
+\sum_{i=4}^{\Delta-1}i(i-1)d_i
-3\Delta\eps-(\Delta-2)\eps'  \nonumber\\
&\leq -\sum_{i=4}^{\Delta-1}(i-3)(\Delta-i)d_i -3\Delta\eps-(\Delta-2)\eps' 
\label{e:quadratic} \\
& \leq -\sum_{i=4}^{\Delta-1}(i-3)(\Delta-i)d_i \leq 0. \nonumber
\end{align}
Hence, the discriminant of the quadratic in $h$ given by \eqref{e:quadl} is 
non-negative, which implies~(\ref{e:lem}).
\end{proof} 

In a linear intersecting $r$-partite hypergraph with $\tau=r$, we have
$4\le\Delta\le r-2$, by \lref{l:maxdegalmdisj}. Substituting
$\Delta\in\{4,5,6\}$ into \eref{e:lem}, yields an immediate contradiction.
Hence:

\begin{corollary}\label{c:Delta<=7}
  If there exists a linear intersecting $r$-partite hypergraph $H$
  such that $\tau(H)=r$, then $\Delta(H)\ge7$.
\end{corollary}

\begin{corollary}\label{c:r<=8} 
  \cjref{conj:Ryser} holds for all linear intersecting $r$-partite
  hypergraphs when $r\leq 8$.
\end{corollary}

\subsection{Linear \boldmath intersecting $9$-partite hypergraphs}\label{s:r=9}

In this subsection we prove that a linear intersecting $9$-partite
hypergraph has covering number at most $8$. For a potential
counterexample $H$, let $h=|H|$ be the number of lines in $H$ and
$\eps:=|V(H)| - r^2\ge0$.  First, we demonstrate some further
properties that such a hypergraph would have, if it were to exist.
Note that we may assume that $\Delta=7$, given \cyref{c:Delta<=7} and
\lref{l:maxdegalmdisj}.

\begin{lemma}\label{l:lower bound on h}
If $H$ is a linear intersecting $9$-partite hypergraph with $\Delta = 7$ and 
$\tau=9$, then $h\ge39$.
\end{lemma}

\begin{proof}
  Let $u$ be a vertex of degree $\Delta=7$, and without loss of
  generality, assume that $u \in V_0$. Let $A$ be the set of vertices
  that lie on lines through $u$ in the remaining eight sides.  Then
  $|A|=8\cdot 7=56$. Let $B = V(H) \setminus (V_0 \cup A)$. Then
  $|B|\geq 16$.

  Let $\ell$ be a line through a vertex in $B$. Then $\ell$ is
  incident with a vertex other than $u$ in side $V_0$ and it
  intersects all seven lines incident with $u$. That is, $|\ell\cap A|
  = 7$ and $|\ell \cap B|=1$.  Since every vertex in $B$ has degree at
  least $2$, there are at least $2|B| \geq 32$ such lines.  Therefore,
  $h \geq 7+32 = 39$.
\end{proof}


By \eref{e:quadratic}, 
\begin{equation}\label{e:quadratic with eps}
  h^2-(\Delta r+ \Delta/2+2r)h+3r^2\Delta \leq - 3\Delta\eps.
\end{equation}

This inequality has no integer solution for $h$ when $\eps\geq 5$ if
$r = 9$ and $\Delta = 7$. Otherwise, together with \lref{l:lower bound on h}, 
we obtain that $h_{min}(\eps) \leq h\leq h_{max}(\eps)$
where $h_{min}(\eps)$ and $h_{max}(\eps)$ are as follows:
\[
\begin{array}{c|ccccc}
  \eps & 0 & 1 & 2 & 3 & 4  \\
  \hline
  h_{min}  & 39 & 39 & 39 & 39 & 42 \\
  h_{max}  & 51 & 50 & 48 & 46 & 42 \\
\end{array}
\]

\begin{lemma}\label{l:deg7s diff sides} 
  If $H$ is a linear intersecting $9$-partite hypergraph with
  $\Delta=7$ and $\tau(H) = 9$, then for every pair of degree $7$ vertices
  on different sides there is a line of $H$ that contains both vertices.
\end{lemma}

\begin{proof}
  Suppose to the contrary that $v_0$ and $v_1$ are degree 7 vertices
  on sides $V_0$ and $V_1$, respectively, which do not lie on a common
  line.  Let $e_1, \dots, e_7$ be the lines which contain $v_0$ and
  $f_1, \dots, f_7$ be the lines which contain $v_1$.  For every $i, j
  \in [1, 7]$, lines $e_i$ and $f_j$ meet at a vertex in
  one of the last seven sides $V_2, \dots, V_8$.

%

Define $V_0'=V_0\setminus\bigl((\cup_i e_i)\cup (\cup_i f_i)\bigr)$ and 
$V_1'=V_1\setminus\bigl((\cup_i e_i)\cup(\cup_i f_i)\bigr)$. Since 
$|V_0|\ge9$, there exists $y\in V_0'$.
Any line $\ell$ through $y$ intersects each of $f_1, \dots, f_7$ in the last 
seven
sides and thus also intersects each of $e_1, \dots, e_7$ in the last seven
sides.  Hence $\ell$ contains a vertex in $V_1'$. By symmetry, all lines 
through $V_1'$ contain a vertex in $V_0'$. Let $G$ be the bipartite graph 
induced on 
$V_0'\cup V_1'$ by $H$.  Then $G$ inherits from $H$ the properties
of having minimum degree at least~$2$ and no line between vertices
of degree~$2$. Therefore $G$ has at least~$5$ vertices and 
at least~$6$ lines. It follows that $\eps \ge 3$ and hence 
$h \leq 46$, from~\eqref{e:quadratic with eps}.

We have already encountered 20 distinct lines of $H$, namely $e_1, \dots e_7,
f_1, \dots, f_7$ and at least $6$ lines through $V_0'$.  Let $B$ be
the set of vertices in the last seven sides which do not lie on any of
these 20 lines. Then $|B| \geq 14$, since $B$ includes at least two
vertices from each of $V_2,\dots,V_8$. Let $x \in B$, and let
$g$ be a line through $x$. Since $g$ intersects each line $e_1, \dots, e_7$, it
follows that $g\cap B=\{x\}$.  However $x$ has degree at least $2$.
Therefore $h \ge 20+2|B| \ge 48$, which is a contradiction.
\end{proof}  

\begin{remark} 
  Although the proof of \lref{l:deg7s diff sides} does not completely
  generalise to larger $r$, it does demonstrate that if $H$ is a
  linear intersecting $r$-partite hypergraph with $\Delta = r-2$,
  $\tau = r$ and $\epsilon \leq 2$, then each pair of degree $\Delta$
  vertices on different sides lie on a common line.
\end{remark}

\subsubsection{Degree sequences, line types, and side types}

Recall that $H$ is assumed to be a linear
intersecting $9$-partite hypergraph with $\tau(H)=r$ and $\Delta=7$.
For each $h$ and $\eps$ within the bounds given by~\eqref{e:quadratic with eps}, we computed the set $\mathcal{D}(h,\eps)$ 
of all possible degree sequences that
satisfy the equations~(\ref{e:di}),~(\ref{e:idi})~and~(\ref{e:iidi}).
Note that since $H$ is linear, equality is enforced
in~(\ref{e:iidi}). From here on, we denote a degree sequence in
$\mathcal{D}(h,\eps)$ by $[d_2,d_3,\dots,d_7]$, where $d_i$ is the number
of vertices of degree $i$, for $i \in [2, 7]$.  For each of the possible values of $h$ and $\eps$, we obtained 
the following number of degree sequences in the set 
$\mathcal{D}(h,\eps)$. An example souce code for this and other 
computational results in this section is posted on the arXiv as an 
ancillary file with the priprint of this article.

\begin{center}
\begin{tabular}{r|rrrrrrrrrrrrr}
$\eps \backslash h$ &  39 &  40 & 41 & 42 & 43 & 44 & 45 & 46 & 47 & 48 & 49 & 50 & 51\\
\hline
0 & 223 & 297 & 307 & 358 & 323 & 311 & 236 & 181 & 107 & 50 & 16 & 2 & 0 \\
1 & 86 & 129 & 135 & 164 & 144 & 145 & 102 & 82 & 42 & 24 & 6 & 0 & -\\
2 & 22 & 39 & 42 & 58 & 48 & 48 & 27 & 20 & 6 & 3 & - & - &  - \\
3  & 1 & 6 & 6 & 11 & 8 & 9 & 3 & 1 & - & - & - & - & - \\
4  & - & - & - & 1 & - & - & - & - & - & - & - & - & -\\
\end{tabular}
\end{center}

Similarly, for given $h$ and $\eps$, we determined the set
$\mathcal{S}(h,\eps)$ of all possible degree sequences of vertices on
a side of $H$, and the set $\mathcal{L}(h)$ of all possible degree sequence
of vertices on a line of $H$. We arbitrarily order sets
$\mathcal{S}(h,\eps)$ and $\mathcal{L}(h)$ to easily index their
elements.  Then for $t \in \{1,2, \dots, |\mathcal{S}(h,\eps)|\}$, a
\emph{side of type} $t$ in $\mathcal{S}(h,\eps)$ is a sequence of
non-negative integers $[s^t_2,s^t_3, \dots ,s^t_7,s^t_\eps]$, in which
$s^t_i$ denotes the number of vertices of degree $i\in[2,7]$,
such that the following conditions hold:
\begin{align*}
(i) &  \hspace{0.3cm} \sum_{i=2}^7 s^t_i = r+s^t_\eps && \mbox{which is the length of a side;  
 } \\
(ii) &  \hspace{0.3cm}  \sum_{i=2}^7 i s^t_i = h && \mbox{since every line meets every side; and  } \\
(iii) & \hspace{0.3cm}  s^t_\eps \leq \eps. 
\end{align*}

A \emph{line of type} $t$ in $\mathcal{L}(h)$, where $t \in \{1,2, \dots, 
|\mathcal{L}(h)|\}$, is a sequence of non-negative integers $[\ell^t_2, 
\ell^t_3,\dots, \ell^t_7]$, in which $\ell_i^t$ denotes the number of vertices 
of degree $i\in[2,7]$, such that the following 
conditions hold:
\begin{align*}
(i) & \hspace{0.3cm} \sum_{i=2}^7  \ell^t_i = r 
&& \mbox{since every line meets every side;} \\
(ii) & \hspace{0.3cm} \sum_{i=2}^7 i \ell^t_i = h+r-1 
&& \mbox{since } H \mbox{ is linear intersecting; and} \\
(iii)	     & \hspace{0.3cm} \ell^t_2\in\{0,1\}
&& \mbox{by \lref{l:tau_r_requirements}(iii). } 
\end{align*}

\subsubsection{Pairwise conditions} \label{sec:pairwise conditions}

Our next goal is to verify which degree sequences are feasible. We
assume that $h$ and $\eps$ are given, and that $D=[d_2,d_3, \dots, d_7]$ is
a degree sequence in $\mathcal{D}(h,\eps)$. Suppose that there exists~$H$,
a linear intersecting hypergraph with $r=\tau=9$, $\Delta=7$ on
$r^2+\eps$ vertices with $h$ lines and the given degree sequence. For
every $t \in \{1,2, \dots, |\mathcal{S}(h,\eps)|\}$, let $x_t$ denote
the number of sides of type $t$ that are contained in $H$.  For every
$t \in \{1,2,\dots, |\mathcal{L}(h)|\}$, let $y_t$ denote the number
of lines of type $t$ that are contained in $H$. Then $x_t$ and $y_t$
are non-negative integers which satisfy some obvious necessary
conditions listed by equations~\eqref{e:side condition 1}--\eqref{e:line condition 1}. Equation~\eqref{e:line condition 2}
is a double count of pairs of vertices of degree $i$ and lines on
which these vertices lie, for $i\in[2,7]$. Since $H$ is a linear
intersecting hypergraph, no pair of vertices is contained on two
lines.  Hence, the number of pairs of vertices contained on a line or
on a side cannot exceed the total number of pairs given by the degree
sequence. This condition is given by equations~\eqref{e:line condition
  3}--\eqref{e:line condition 5}, depending on whether we count pairs
of vertices having the same degree, or pairs of vertices with
different degree.  Putting all of these conditions together, we
formulate an integer program on variables $x_t$ and $y_t$.

\begin{align}
  & \sum_t x_t = r && \mbox{there are } r \mbox{ sides } \label{e:side 
condition 1} \\
  & \sum_t x_t s_\eps^t = \eps && \mbox{there are $\eps$ extra 
vertices}  \label{e:side 
condition 2} \\
  & \sum_t x_t s^t_i = d_i && \mbox{there are $d_i$ vertices of degree $i 
\in [2,7]$}  \label{e:side 
condition 3} \\ 
 & \sum_t y_t = h && \mbox{there are $h$ lines} \label{e:line condition 1} \\
 & \sum_t y_t \ell^t_i = i d_i && i \in 
[2,7] \label{e:line condition 2} \\
 & \sum_t x_t {s^t_i \choose 2} + \sum_t y_t {\ell^t_i \choose 2} \leq {d_i 
\choose 2} && i \in [2,6] 
\label{e:line condition 3}  \\
 & \sum_t x_t {s^t_7 \choose 2} + \sum_t y_t {\ell^t_7 \choose 2} = {d_7 
\choose 2} && \mbox{by \lref{l:deg7s diff sides}} \label{e:line 
condition 4}  \\
& \sum_t x_t s^t_{i} s^t_{j} + \sum_t y_t \ell^t_{i} 
\ell^t_{j} \leq d_{i} d_{j} && \mbox{distinct }i, j \in [2,7]
\label{e:line condition 5} 
\end{align}

This integer program has a feasible solution only for the following twelve 
degree sequences $D=[d_2,d_3,\dots, d_7]$. For all of these, $\eps = 0$.  
\[\begin{array}{lcl}
h= 45 \; \;  D= [22, 3, 7, 2, 15, 32], & \quad &
h= 46 \; \;  D= [23, 2, 1, 7, 13, 35], \\
h= 46 \; \;  D= [23, 1, 3, 7, 11, 36], & \quad &
h= 46 \; \;  D= [23, 0, 6, 4, 12, 36], \\
h= 46 \; \;  D= [23, 0, 5, 7, 9, 37], & \quad &
h= 46 \; \;  D= [22, 2, 5, 5, 10, 37], \\
h= 46 \; \;  D= [23, 0, 4, 10, 6, 38], & \quad &
h= 46 \; \;  D= [22, 1, 7, 5, 8, 38], \\
h= 46 \; \;  D= [22, 0, 10, 2, 9, 38], & \quad &
h= 47 \; \;  D= [23, 0, 3, 5, 10, 40], \\
h= 47 \; \;  D= [22, 1, 4, 6, 6, 42], & \quad &
h= 47 \; \;  D= [21, 2, 6, 4, 5, 43]. \\
\end{array}
\]

\subsubsection{Assignment of lines to vertices}

If a linear intersecting $9$-partite hypergraph $H$ with $\tau(H)=9$
exists, then it has $81$ vertices ($\eps=0$), $h \in \{45,46,47\}$,
and one of the twelve degree sequence listed at the end of
\sref{sec:pairwise conditions}. Let $D$ be one of these degree
sequences. Next we generated the set $\mathbf{S}(D)$ of all
possible non-negative integer solutions for the system of
equations~(\ref{e:side condition 1}),~(\ref{e:side condition
  2}),~and~(\ref{e:side condition 3}). Let $S=[x_1, x_2, \dots,
x_{|\mathcal{S}(h,0)|}]$ denote a particular solution in $\mathbf{S}(D)$.  We
formulate another system of linear equations which takes $D$ and
$S$ as input values.

Let $V=V(H)$ be the vertex set of $H$ which is partitioned into $9$
sides of equal size. Then the set of degree sequences of vertices in
the sides of $H$ corresponds to a solution $S \in \mathbf{S}(D)$.  We
change the notation slightly, to let $s^v_j$ denote the number of
vertices of degree $j$ in the side that contains vertex $v$.

As before, let $y_t$ be the number of lines of type $t$ from the set
$\mathcal{L}(h)$ present in $H$. Define $z_t^v$ to be the number
of lines of type $t$ incident with a vertex $v \in V$. If a line of
type $t$ has no vertices of degree $\deg(v)$ then
$z_t^v=0$. Otherwise, $z_t^v$ is a non-negative integer.  In addition
to equations \eref{e:line condition 1}--\eref{e:line condition 5}, the
following equations hold for $y_t$ and $z_t^v$, where $t \in \{1,2,
\dots, |\mathcal{L}(h)|\}$ and $v \in V$.


\begin{align}
&\quad\sum_t z_t^v = \deg(v) && \mbox{for all $v \in V$}; \label{e:line types to 
degs of v}\\
&\quad\sum_t z_t^v (\ell^t_i-1) \leq d_i - s^v_i && \mbox{for all $v \in V$ 
where $i=\deg(v)$}; \label{e:seen vertices of the same degree} \\
&\quad\sum_t z_t^v (\ell^t_7 -1 ) = d_7 - s^v_7 && \mbox{for all $v \in V$ 
such that $\deg(v)=7$;} \label{e:seen vertices of the same degree 7} \\
&\quad\sum_t z^v_t \ell^t_j \leq d_j - s^v_j && \mbox{for all $v \in V$ and 
all $j \in [2,7]$, $j \neq \deg(v)$;} \label{e:seen vertices of different 
degree} \\
&\quad\sum_{\kern-2em v \in V_k\kern-2em} z^v_t= y_t && \mbox{for all 
$t \in \{1,2,\dots, |\mathcal{L}(h)|\}$ and all $k \in [1,r]$;} 
\label{e:line types in sides}\\ 
&\quad\sum_{\kern-2em v, \deg(v)= i\kern-2em} z_t^v = y_t \ell^t_i && \mbox{for all $i \in 
[2,7]$ and $t \in \{1,2, \dots,|\mathcal{L}(h)|\}$.} \label{e:double 
count lines and degrees}
\end{align}

Note that since $S \in \mathbf{S}(D)$ is given, now values $x_t$ in
equations~(\ref{e:line condition 1})-(\ref{e:line condition 5}) are
constants. The total number of lines incident with a vertex equals the
degree of that vertex, which is given by \eref{e:line types to degs of v}. 
Two vertices are \emph{neighbours} if there is a line containing
 both of them.  Observe that, for each vertex $v$ in a given side, the
number of neighbours of $v$ which have degree $j$ is at most the total
number of degree $j$ vertices in the remaining sides, for $j \in
[2,7]$. Equations~(\ref{e:seen vertices of the same
degree})-(\ref{e:seen vertices of the same degree 7}) correspond to
counting the number of neighbours of $v$ which are of the same degree
as~$v$, whereas~(\ref{e:seen vertices of different degree}) counts the
neighbours of $v$ which have degree different from $\deg(v)$.  The
equality in~(\ref{e:seen vertices of the same degree 7}) is implied by
\lref{l:deg7s diff sides}. Since every line contains exactly one
vertex in each side, each side is incident with as many
lines of a type $t$ as there are lines of type $t$ present in the
hypergraph, which is given by~(\ref{e:line types in
  sides}). Finally,~(\ref{e:double count lines and degrees}) is a
double count of pairs of lines of type $t$ and vertices of degree $i$
incident with these lines.

We found a feasible solution for the integer program given 
by~\eqref{e:line condition 1}--\eqref{e:double count lines and degrees} only 
for  two pairs of input 
values of a degree sequence $D$ and $S \in \mathbf{S}(D)$ which 
we consider more closely in the following subsection.

\subsubsection{Remaining cases}

The two cases for which the integer program in the previous section
has a feasible solution both have $h = 46$.  Below we give the input
degree sequence $D$ and a matrix representation of $S$. Here, a column
of $S$ corresponds to a side, and each entry is the degree of a vertex
in that side. Each matrix $S$ is given uniquely, up to permutation of
sides and permutation of vertices within each side.
\begin{align*}
&\textrm{\textbf{Case 1}} 	&& \textrm{\textbf{Case 2}}\\
&  D = [23,1,3,7,11,36]   &&  D = [23,0,5,7,9,37]\\
& S = \left( \footnotesize  \begin{array}{rrrrrrrrr}
2 & 2 & 2 & 2 & 2 & 2 & 2 & 2 & 2 \\
2 & 2 & 2 & 2 & 2 & 2 & 2 & 2 & 2 \\
4 & 4 & 4 & 2 & 2 & 3 & 2 & 2 & 2 \\
5 & 5 & 5 & 6 & 6 & 5 & 5 & 5 & 5 \\
6 & 6 & 6 & 6 & 6 & 6 & 7 & 7 & 7 \\
6 & 6 & 6 & 7 & 7 & 7 & 7 & 7 & 7 \\
7 & 7 & 7 & 7 & 7 & 7 & 7 & 7 & 7 \\
7 & 7 & 7 & 7 & 7 & 7 & 7 & 7 & 7 \\
7 & 7 & 7 & 7 & 7 & 7 & 7 & 7 & 7
\end{array}\right) 
&& S = \left( \footnotesize \begin{array}{rrrrrrrrr}
2 & 2 & 2 & 2 & 2 & 2 & 2 & 2 & 2 \\
2 & 2 & 2 & 2 & 2 & 2 & 2 & 2 & 2 \\
4 & 5 & 2 & 2 & 4 & 4 & 2 & 2 & 2 \\
5 & 5 & 6 & 6 & 4 & 4 & 5 & 5 & 5 \\
6 & 5 & 6 & 6 & 6 & 6 & 7 & 7 & 7 \\
6 & 6 & 7 & 7 & 7 & 7 & 7 & 7 & 7 \\
7 & 7 & 7 & 7 & 7 & 7 & 7 & 7 & 7 \\
7 & 7 & 7 & 7 & 7 & 7 & 7 & 7 & 7 \\
7 & 7 & 7 & 7 & 7 & 7 & 7 & 7 & 7
\end{array}\right)
\end{align*}

\medskip
\noindent \textbf{Case 1:}  Suppose there exists a linear intersecting $9$-partite hypergraph $H$ with $\tau(H)=9$, $h=46$ lines and degree sequence $D = [23,1,3,7,11,36]$. First, we consider the 
set of types of lines in $H$.
Since $H$ has $46$ lines and $d_2=23$, every line in $H$ contains a 
vertex of degree $2$. 
Also, $\ell_i^t \leq d_i$ for all $i \in [3,7]$. Hence, the 
possible line types $\mathcal{L}$ in $H$ are
$$ \mathcal{L} = \{ [1,0,0,0,4,4], [1,0,0,1,2,5], [1,0,0,2,0,6], [1,0,1,0,1,6], 
[1,1,0,0,0,7]\}\subset \mathcal{L}(46).$$ 
Note that $H$ has one vertex $v$ such that $\deg(v)=3$. Let $e_1$ and
$e_2$ be two lines which contain $v$. Then $e_1$ and $e_2$ have the
same line type, namely $[1,1,0,0,0,7]$, since this is the only line
type in $\mathcal{L}$ which contains a vertex of degree $3$.  Let
$u_1$ and $u_2$ be the vertices of degree $2$ on $e_1$ and $e_2$,
respectively. If $u_1$ and $u_2$ belong to two different sides, then
let $w$ be the vertex of degree $7$ on $e_2$ which is on the same side
as $u_1$. By \lref{l:deg7s diff sides}, there is a distinct line through $w$ 
and each of the degree $7$ vertices on $e_1$. Hence, $\deg(w) \geq 8$,
which is a
contradiction. Therefore, $u_1$ and $u_2$ are on the same side in
$H$. Let $f$ be the line other than $e_1$ which contains $u_1$ and
let $w'$ be the vertex in which $e_2$ and $f$ intersect. Then
$\deg(w')=7$. Again by \lref{l:deg7s diff sides}, there is a line
through $w'$ and every degree $7$ vertex on $e_1$ which is not on the
same side as $w'$. Hence, $\deg(w') \geq 2+6=8$, which is a
contradiction.  Therefore, such a hypergraph does not exist.

\medskip
\noindent \textbf{Case 2:} Suppose there exists a linear intersecting $9$-partite hypergraph $H$ with $\tau(H)=9$, $h=46$ lines and degree sequence $D=[23,0,5,7,9,37]$. Moreover, without loss of generality, assume
that the degrees of vertices in $V=V(H)$ are given by $S$ for a
suitable permutation of sides and permutation of vertices within each
side.  Let $\mathcal{L} \subseteq \mathcal{L}(46)$ be the set of all
possible line types which may occur in $H$. As in the previous case,
line types in $\mathcal{L}$ satisfy the obvious constraints on
degrees. Hence
$$\mathcal{L} = \{ [1,0,0,0,4,4], [1,0,0,1,2,5], [1,0,0,2,0,6], [1,0,1,0,1,6]\} 
.$$

For clarity, we index the line types in $\mathcal{L}$ by  $A$, 
$B$, $C$ and $D$, respectively. Then for $t \in \{A,B,C,D\}$,  $y_t$ is the 
number of lines of type $t$ in $H$ and  $(y_A, y_B, y_C, y_D)$
satisfies the equations~\eqref{e:line condition 1}--\eqref{e:line condition 
5}. In this case, many of these equations are dependent and it is enough to 
consider 
equations~\eqref{e:line condition 1},~\eqref{e:line condition 2} for $i \in 
\{4,5\}$, and~\eqref{e:line condition 4} to obtain the
unique solution $(y_A, y_B, y_C, y_D) = (1,15,10,20)$. 

Moreover, for each $t \in \{A,B,C,D\}$ and each $v \in V$, $z^v_t$ denotes the number of 
lines of type $t$ incident with the vertex $v$ and the set of all values $\{ z_t^v \; : \; v\in V, t 
\in \{A,B,C,D\} \}$ satisfies equations~\eqref{e:line types to 
degs of v}--\eqref{e:double count lines and degrees}. We restrict our attention 
to vertices of degree $7$ and compute all possible solutions to the system of 
equations given only by~\eqref{e:line types to degs of v},~\eqref{e:seen 
vertices of the same degree 7} and~\eqref{e:seen vertices of different degree} 
for $j \in [2,6]$. 
For each $0\le k\le 8$, the possible solutions $(z^v_A, z^v_B, z^v_C, z^v_D)$ 
for $v \in V_k$, where $\deg(v)=7$, are given in the table below.
\begin{center}
\begin{tabular}{c|c|c|c}
$V_0$ & $V_1$ & $V_2, V_3, V_4, V_5$ & $V_6, V_7, V_8$ \\
\hline
$(0,1,2,4)$ & $(0,1,1,5)$ & $(1,0,3,3)$ & $(1,1,2,3)$ \\
             &            & $(0,2,2,3)$ & $(0,3,1,3)$ \\
            
\end{tabular}
\end{center}

Let $e$ be a line of type $C$. Line $e$ intersects each of the 20 lines of type $D$ 
exactly once. Moreover, line $e$ intersects a line of type $D$ either in its vertex of 
degree $2$ or in one of its 6 vertices of degree $7$. Observe that vertices of degree $7$ in sides $V_0$ and $V_1$ are 
incident with $4$ and $5$ lines of type $D$, respectively; all other 
vertices of degree $7$ are always incident with exactly $3$ lines of type $D$.  Depending on whether $e$ 
intersects a line of type $D$ in its vertex of degree $2$ or 
not, it is easy to see that, in order for $e$ to meet $20$ lines of type $D$,  
$e$ either contains exactly one vertex of degree 
$7$ on side $V_0$ or exactly one vertex of degree $7$ on side $V_1$, 
but not both. Since $z_C^v = 2$ if $v\in V_0$ and $z_C^v=1$ if $v \in V_1$ 
when $\deg(v)=7$, there are at most $2 \cdot 3 + 1 \cdot 3 = 9$ lines of type 
$C$, which gives a contradiction.  Therefore, such a hypergraph does not exist.

We conclude that there does not exist a linear intersecting
$9$-partite hypergraph with covering number $9$ and $\Delta = 7$.
Together with Corollary~\ref{c:r<=8} we have shown:

\begin{theorem}\label{thm: r <= 9}
  For $2 \leq r \leq 9$ every linear intersecting $r$-partite
  hypergraph has covering number at most $r-1$.
\end{theorem}

\section{Hypergraph constructions}\label{sec:constructions}

In this section we describe several hypergraph constructions that are
of interest.  First we define some additional notation.
Let $H$ be an $r$-partite hypergraph.  In each side $V_k$, we label
the vertices by $(k,0),(k,1),\dots,(k,|V_k|-1)$, which we
abbreviate to $0,1,\dots,|V_k|-1$ when the side is clear from context.
We say that a vertex $(k,l)\in V_k$ is at \emph{level} $l$ in side
$V_k$.  We denote a line $e$ in $H$ by $[l_0,l_1, l_2, \dots
l_{r-1}]$, where $e$ contains the vertex at level $l_k$ in side $V_k$,
where $k \in [0,r-1]$. When presenting a specific example, we omit the square
brackets and commas to make the notation cleaner. The \emph{cyclic
  $1$-shift} of a line $[l_0, l_1, l_2,\dots, l_{r-1}]$ is the line
$[l_{r-1}, l_0, l_1, \dots, l_{r-2}]$ and the \emph{cyclic $t$-shift}
of a line $e$, denoted $e^t$, is the line obtained from $e$ by
applying $t$ cyclic $1$-shifts.  In particular, $e^0=e=e^r$.  An
$r$-partite hypergraph is \emph{cyclic} if its automorphism group
contains a cyclic subgroup of order $r$ acting transitively on the
sides.  A cyclic $r$-partite hypergraph can be obtained by developing
a set of \emph{starter lines} by cyclic shifts.

\subsection{A  \boldmath counterexample to \cjref{conj:variation}}\label{sec:counterex}

In this subsection we give a counterexample to \cjref{conj:variation}. 

\begin{theorem}\label{thm:counterexABW} 
  For at least one value of $r$ there is an intersecting
  $r$-partite hypergraph $H$ such that each of its sides has size $r$
  and $e \setminus \{v\}$ is not a cover for any $e \in H$ and any $v
  \in e$.
\end{theorem}

\begin{proof}
  Let $r=13$.
  We give an example of a cyclic linear intersecting $r$-partite
  hypergraph $H$ in which every side has size $r$. The $3r$
  lines of $H$ are obtained by taking all possible cyclic shifts of
  the following three starter lines.
\begin{equation}\label{e:e1e2e3}
\begin{array}{rccccccccccccc}
e_1 =& 0 & \mathbf{1} & 4 & \mathbf{2} & \mathbf{3} & 5 & 6 & 6 & 5 & 
\mathbf{3} & \mathbf{2} & 4 & \mathbf{1} \\
e_2 =& 0 & \mathbf{3} & 7 & \mathbf{1} & \mathbf{2} & 8 & 9 & 9 & 8 & 
\mathbf{2} & \mathbf{1} & 7 & \mathbf{3} \\
e_3 =& 0 & \mathbf{2} & 10 & \mathbf{3} & \mathbf{1} & 11 & 12 & 12 & 11 & 
\mathbf{1} & \mathbf{3} & 10 & \mathbf{2}
\end{array}
\end{equation}
Vertices at level $0$ in $H$ have degree $3$, vertices at levels $1$,
$2$ and $3$ have degree $6$, and all other vertices have degree $2$.

We claim that $H$ is a linear intersecting hypergraph. Observe
that $|e_i^{t_1} \cap e_j^{t_2}|=|e_i \cap e_j^{t_2-t_1}|$ 
for any $i,j \in \{1,2,3\}$ and $0\le t_1 \leq t_2< r$. Hence, it
suffices to show that for any $i,j \in \{1,2,3\}$ and $0\le t<r$
where $(j,t) \neq (i,0)$, lines $e_i$ and $e_j^t$ intersect in exactly one
vertex.

First we consider the case when $i=j$.  By construction, $e_i^3$ has
the underlying structure of a 4-extended Skolem sequence of order $6$,
namely 6420246531135.  For example, this sequence is obtained from
$e_1^3$ by relabelling the levels using the permutation
$(1,2,6)(3,5)$.  For definitions and background on Skolem
sequences, see~\cite{skolem_survey}.  In our case, the result is that
for every $t \in [1,\frac{r-1}{2}]$, there is a unique pair $(k,l)$
and $(k',l)$ of vertices in $e_i$ with $k'-k\equiv t\mod r$. Hence
$e_i\cap e_i^t=\{(k',l)\}$ and $e_i\cap e_i^{r-t}=\{(k,l)\}$.

Now assume that $i \neq j$. Suppose that there are two 
distinct vertices $(k_1,l_1)$ and $(k_2,l_2)$ in $e_i\cap e_j^t$.
Since $e_i\cap e_j=\{(0,0)\}$, we may assume that $t\in[1,r-1]$.
Then by inspection, we must have $l_1,l_2\in\{1,2,3\}$ (the relevant
entries are shown in {\bf bold} in \eref{e:e1e2e3}).
For the different possible pairs $(l_1,l_2)$ the following table shows
the feasible distances $k_2-k_1\mod r$. 
\[
\begin{array}{c|| c | c | c} 
 \hspace{0.5cm} l_1, l_2 \hspace{0.5cm} & \textrm{distances in $e_1$} & 
\textrm{distances 
in $e_2$} & \textrm{distances in $e_3$}\\ \hline 
 1,1 & \pm2& \pm6& \pm5\\
 2,2 & \pm6& \pm5& \pm2\\
 3,3 & \pm5& \pm2& \pm6\\
 1, 2 & \pm 2, \pm 4 &  \pm 1, \pm 6 &  \pm 3, \pm 5  \\
 1, 3 & \pm 3, \pm 5 &  \pm 2, \pm 4 &  \pm 1, \pm 6 \\
 2, 3 & \pm 1, \pm 6 &  \pm 3, \pm 5 &  \pm 2, \pm 4 \\ 
\end{array}
\]
Since there is no row of the table where the same distance occurs in
different columns, we conclude that 
\begin{equation}\label{e:eiej}
|e_i\cap e_j^t|\le 1.
\end{equation}
Since $e_i$ and $e_j$ intersect in the vertex $(0,0)$, 
and both have a pair of vertices on each level $l \in
\{1,2,3\}$, there are $1 + 3 \cdot 4 = 13$ vertices in which $e_i$
intersects the set of all cyclic shifts of $e_j$.  It follows
that we must have equality in \eref{e:eiej} for each $i,j,t$.
Thus $H$ is a linear intersecting hypergraph.
Since $\delta(H)=2$, it follows that for
every line $e \in H$ and each vertex $v \in e$, the set of vertices 
$e\setminus \{v\}$ is not a cover.
\end{proof}

For the linear intersecting $13$-partite hypergraph constructed in the proof of 
\tref{thm:counterexABW} we found by computation that $\tau(H)=9$.
The following  set of vertices is a 
$9$ cover for $H$: $$\{ (0,1), (0,2), (0,3), 
(2,0), (5,0), (8,0), (10,1), (10,2), (10,3)\}.$$ 

\subsection{Cyclic \boldmath intersecting hypergraphs 
with $\tau = r-1$ 
}\label{sec:non_prime_power}

Next we give an example of a cyclic intersecting $r$-partite
hypergraph which has covering number $r-1$ for $r \in
\{9,13,17\}$. The methodology for building these hypergraphs is
similar to the construction given in the proof of
\tref{thm:counterexABW}. However, in each case below the hypergraph we
construct is non-linear.

\begin{lemma}\label{l:f(r)}
Let $(r,s) \in \{(9,4), (13,5), (17,6)\}$. Then there exists a cyclic 
intersecting $r$-partite hypergraph $H$ such that $H$ has $sr$ lines and 
$\tau(H)=r-1$.  
\end{lemma}
\begin{proof}
  The $sr$ lines of $H$ are obtained by taking all possible cyclic
  shifts of lines in the starters given below.
\begin{align*}
 \begin{array}{lccccccccc}
\multicolumn{10}{l}{r=9} \\ \hline \hline 
e_1 = & 4 & 3 & 2 & 1 & 0 & 1 & 2 & 3 & 4 \\
e_2 = &3 & 6 & 5 & 4 & 0 & 4 & 5 & 6 & 3 \\
e_3 = &1 & 2 & 4 & 6 & 0 & 6 & 4 & 2 & 1 \\
e_4 = &3 & 0 & 2 & 6 & 7 & 6 & 2 & 0 & 3 \\
&  &  &  &  &  &  &  &  &  \\
\end{array}
&&
\begin{array}{lccccccccccccc}
\multicolumn{14}{l}{r=13} \\ \hline \hline 
e_1 =& 6 & 5 & 4 & 3 & 2 & 1 & 0 & 1 & 2 & 3 & 4 & 5 & 6 \\
e_2 =& 9 & 6 & 1 & 8 & 7 & 5 & 0 & 5 & 7 & 8 & 1 & 6 & 9 \\
e_3 =& 7 & 2 & 9 & 1 & 3 & 10 & 0 & 10 & 3 & 1 & 9 & 2 & 7 \\
e_4 =& 5 & 3 & 2 & 6 & 9 & 7 & 0 & 7 & 9 & 6 & 2 & 3 & 5 \\
e_5 =& 9 & 5 & 2 & 6 & 3 & 7 & 11 & 7 & 3 & 6 & 2 & 5 & 9\\
\end{array}
\end{align*}
 
\[
 \begin{array}{lccccccccccccccccc}
 \multicolumn{18}{l}{r=17} \\ \hline \hline 
e_1 = & 3 & 12 & 11 & 4 & 1 & 10 & 2 & 9 & 0 & 9 & 2 & 10 & 1 & 4 & 11 & 12 & 3 
 \\
e_2 = & 1 & 8 & 7 & 6 & 5 & 4 & 3 & 2 & 0 & 2 & 3 & 4 & 5 & 6 & 7 & 8 & 1  \\
e_3 = & 14 & 3 & 9 & 7 & 15 & 11 & 6 & 4 & 0 & 4 & 6 & 11 & 15 & 7 & 9 & 3 & 
14 \\
e_4 = & 6 & 2 & 5 & 11 & 14 & 12 & 4 & 10 & 0 & 10 & 4 & 12 & 14 & 11 & 5 & 2 & 
6\\
e_5 = & 11 & 10 & 14 & 1 & 7 & 3 & 12 & 8 & 0 & 8 & 12 & 3 & 7 & 1 & 14 & 10 & 
11 \\
e_6 = & 7 & 4 & 13 & 5 & 9 & 1 & 11 & 12 & 0 & 12 & 11 & 1 & 9 & 5 & 13 & 4 & 7
\\
\end{array}
\]
We found by computation that the covering numbers for these three hypergraphs 
are $8$, $12$ and $16$, respectively.
\end{proof}

\begin{remark}
The hypergraphs with $r=9$, $r=13$ and $r=17$ in \lref{l:f(r)}
are intrinsically non-linear; each has the property that it does not
contain a linear subhypergraph with covering number $r-1$.

First, if $H$ is the hypergraph with $r=9$, then 
$|e_2 \cap e_4| = 2$, but deleting either $e_2$ or $e_4$ reduces the covering 
number. Below are covers of size $7$ for $H \backslash \{e_2\}$ and $H 
\backslash \{e_4\}$, respectively.
\begin{align*}
& \{( 5 , 6 ), ( 6 , 0 ), ( 6 , 1 ), ( 6 , 2 ), ( 6 , 3 ), ( 6 , 4 ), ( 6 , 6 ) \} \\
& \{( 4 , 0 ), ( 4 , 1 ), ( 4 , 2 ), ( 4 , 3 ), ( 4 , 4 ), ( 4 , 5 ), ( 4 , 6 )\}
\end{align*}

Similarly, if $H$ is the hypergraph with $r=13$, then 
$|e_1 \cap e_5| = 2$ and below are covers of size 
$11$ for $H \backslash \{e_1\}$ and $H \backslash \{e_5\}$, respectively.
\begin{align*}
& \{( 0 , 7 ), ( 1 , 3 ), ( 2 , 3 ), ( 3 , 7 ), ( 8 , 0 ), ( 8 , 1 ), ( 8 , 3 ), (8 , 5 ), ( 8 , 6 ), ( 8 , 7 ), ( 12 , 2 )\} \\
& \{( 1 , 9 ), ( 4 , 6 ), ( 6 , 0 ), ( 6 , 1 ), ( 6 , 2 ), ( 6 , 3 ), ( 6 , 6 ), (6 , 7 ), ( 6 , 9 ), ( 8 , 6 ), ( 11 , 9 )\}
\end{align*}
\end{remark}

Finally, consider the hypergraph $H$ in Lemma~\ref{l:f(r)} with $r=17$. Let 
$H_1$ be the subhypergraph obtained from all of the cyclic shifts of starter 
lines $e_2, e_3, e_4, e_5, e_6$.  Then $H_1$ is a 
linear intersecting hypergraph with covering number 15, where the following set of vertices is a minimal cover:
\begin{align*}
\{ & ( 4 , 5 ), ( 4 , 8 ), ( 4 , 11 ), ( 5 , 4 ), ( 5 , 7 ), ( 10 , 0 ), ( 10 , 4 ), ( 10 , 6 ),\\
& ( 10 , 7 ), ( 10 , 14 ), ( 15 , 4 ), ( 15 , 7 ), ( 16 , 5 ), ( 16 , 8 ), ( 16 , 11 )\}. 
\end{align*}
Let $E_1$ be the set of all cyclic shifts of the starter line $e_1$.
Observe that $E_1$ is linear intersecting.  Furthermore, each line
$e\in E_1$ intersects exactly 12 lines of $H_1$ more than once.  For
every non-empty subset $E^* \subseteq E_1$, define $H_1^*$ to be the
hypergraph obtained from $H_1$ by adding the lines of $E^*$ and
removing every line of $H_1$ that intersects a line of $E^*$ more than
once.  We computationally checked that each of the $2^{17}-1$ such
linear intersecting hypergraphs $H_1^*$ has covering number less than
16 (in fact, they each have covering number between 10 and 14).
Therefore, $H$ has no linear intersecting subhypergraph with $\tau=16$.

\subsection{Linear intersecting hypergraphs built from latin squares}\label{sec:mols}

Next we describe a family of linear intersecting $r$-partite hypergraphs which do not have a minimal cover that consists of a single side or that consists of a subset of a line.

A \emph{latin square} of order $n$ is an $n\times n$ array of $n$ symbols such 
that each symbol appears exactly once in each row and exactly once in each 
column.  If $L$ is a latin square, $L[r,c]$ denotes the symbol in row $r$ and 
column $c$ of $L$.  A pair of latin squares of order $n$ are \emph{orthogonal} 
if, when the two squares are superimposed, each of the $n^2$ possible ordered 
pairs of symbols occurs exactly once.  A set of latin squares is called 
\emph{mutually orthogonal} if each pair of latin squares in the set are 
orthogonal. For more on latin squares, see~\cite{vLW01}.

\begin{lemma}\label{l:fromMOLS_2extra_sides} 
  If there exist $k$ mutually orthogonal latin squares of order
  $n\ge3$, then there exists a linear intersecting $(n+2)$-partite
  hypergraph $H$ with $\tau(H) = k+1$ such that no side or subset of a
  line is a minimal cover.
\end{lemma}

\begin{proof}
Let $L_0, \dots, L_{k-1}$ be $k$ mutually orthogonal latin squares of order $n$.  Assume that each $L_i$  has $[0,1,\dots, n-1]$ as its first row.  For $i = 0,1,\dots, k-1$, let $M_i$ be the column inverse of $L_i$, that is, $M_i[r,c] = s$ if and only if $L_i[s,c] = r$.
Since each $L_i$ has its first row in reduced form, each $M_i$ has symbol $0$ on the main diagonal.  Set $M_0' = M_0$ and for $i = 1, \dots, k-1$ let $M_i'$ be the latin square obtained from $M_i$ by replacing symbol $0$ with symbol $n+i-1$.

Next we define the lines of an $(n+2)$-partite hypergraph $H$.  Let
$V_0, \dots, V_{n+1}$ be the sides of $H$.  Sides $V_0, \dots, V_{n}$
each have $n+k-1$ vertices and side $V_{n+1}$ has $n+1$ vertices.
If a line $[l_0, l_1, \dots, l_{r-1}]$ is a concatenation of two
lists, we use the notation $[l_0, \dots, l_i] \oplus [l_{i+1}, \dots,
l_{r-1}]$.

For $i = 0, \dots, k-1$ define 
$E_i := \{  M_i'[x] \oplus [i, x] : 0\le x\le n-1 \}$, 
where $M_i'[x]$ is row $x$ of the latin square $M_i'$.
Further, define 
$E^* := \{ [x,x,\dots, x, k+x-1, n]: 1\le x\le n-1\} $.

It is straightforward to check that each line in $E_i$ intersects each
line in $E^*$ exactly once, and that the lines in $E^*$ meet each
other exactly once.  Since $M_i'$ is a latin square, two distinct
lines in $E_i$ intersect only at the vertex $(n,i)$.  We next show
that lines of $E_i$ and $E_j$ also intersect linearly for $i \neq j$.
Consider a line $\ell_1 = M_i'[x] \oplus [i,x] \in E_i$ and a line
$\ell_2 = M_j'[y] \oplus [j,y] \in E_j$ where $i \neq j$.  Since $L_i$
and $L_j$ are orthogonal latin squares, there is a unique cell $(r,c)$
where $L_i[r,c] = x$ and $L_j[r,c] = y$.  If $x=y$, then $r=0$ and
$M_i[x,x]=M_j[x,x]=0$, and thus, by the relabelling of symbol $0$ in
$M_1', \dots, M_{k-1}'$, the lines $\ell_1$ and $\ell_2$ intersect
only at the vertex $({n+1},x)$.  If $x \neq y$, then $r \neq 0$ and
thus $M_i'[x,c]=r=M_j'[y,c]$ and therefore the lines $\ell_1$ and
$\ell_2$ intersect only at the vertex $(c,r)$.  It follows that $H =
E^* \cup E_0 \cup E_1 \cup \dots \cup E_{k-1}$ is a linear
intersecting $(n+2)$-partite hypergraph with $kn + n-1$ lines.

Since $H$ has maximum degree $n$ and $kn+n-1$ lines, no set of $k$
vertices is a cover.  It is straightforward to check that the vertices
at levels $0,1,\dots,k-1$ of side $V_{n}$ together with the vertex
$({n+1},n)$ form a cover of size $k+1$.  Therefore $\tau(H) = k+1$.

It is well-known that there are at most $n-1$ mutually orthogonal
latin squares of order $n$ (see e.g.~\cite{vLW01}), so $k+1 \leq n$. 
Observe that for $k\ge2$ each side has size at least $n+1$, whereas 
$\tau(H)=2<n$ when $k=1$. Hence,
no side is a minimal cover.  Also, since $H$ is linear and each line
contains exactly one vertex of degree $1$, it follows that the only
covers which are subsets of a line have size at least $n+1$.  Thus, no
side or subset of a line is a minimal cover of $H$.
\end{proof}

The following corollary follows immediately from the existence of complete sets 
of mutually orthogonal latin squares of prime power orders~\cite{vLW01}.

\begin{corollary}
If $n$ is a prime power then there exists a linear intersecting $(n+2)$-partite hypergraph $H$ with $\tau(H) = n$ such that no side or subset of a line is a minimal cover.
\end{corollary}

Example: Below are the lines of a hypergraph built from $3$ mutually 
orthogonal latin squares of order $4$, as described in 
\lref{l:fromMOLS_2extra_sides}.  This hypergraph is also maximal with respect to the property of being linear and intersecting.
\[\begin{array}{llll}
031200 & \quad 423110 & \quad 512320 & \quad 111134\\
302101 & \quad 241311 & \quad 153221 & \quad 222244\\
120302 & \quad 314212 & \quad 235122 & \quad 333354\\
213003 & \quad 132413 & \quad 321523 & \quad \\
\end{array}\]

\subsection{An \boldmath $8$-partite linear hypergraph with $\tau=7$}\label{s:H38}

We close this section by giving a construction for an interesting
$8$-partite linear hypergraph $\HH$ that achieves equality in
\cjref{conj:Ryser}.

Let $F$ denote the Fano plane constructed by developing the triple
$\{0,1,3\}$ modulo $7$. Let $G$ denote the stabiliser of the point
$0$ in $F$. Note that $|G|=24$.
Let $\C$ denote the set of $7$-cycles obtained by conjugating
the cycle $(0123456)$ by elements of $G$. 

For each permutation $p\in\C$ we add one line to $\HH$ which includes
the vertex $(i,p[i])$ for $i\in[0,6]$.
If $p_1,p_2\in\C$ then $p_1^{-1}p_2$ has at most one fixed point.
We can make the lines corresponding to $p_1$ and $p_2$ meet on
side $V_7$ if and only if $p_1^{-1}p_2$ has no fixed points.
This produces $8$ vertices of degree $3$ in $V_7$, and completes
the description of the lines corresponding to the cycles in $\C$.
Next we add two new vertices $v_1,v_2$ to $V_7$.
For each $i\in[0,6]$ we put a line through $v_1$ and the vertices 
$(j,i)$ for $j\in[0,6]$. For each
$i\in[0,6]$ we put a line through $v_2$, $(i,i)$ and all vertices 
$(a,b)$ for which $\{i,a,b\}$ is a triple of $F$.

The construction just described results in $\HH$, which has $38$ lines
and an automorphism group isomorphic to PSL$(2,7)$. For $i\in[0,6]$
the vertex $(i,i)$ has degree $2$. All other vertices on sides
$V_0,\dots,V_6$ have degree $6$. On $V_7$ the vertices $v_1$ and $v_2$
have degree $7$ and the other $8$ vertices have degree $3$.  Since
$V_7$ has 10 vertices it is clear that $\HH$ is not isomorphic to
a subhypergraph of $\P'_8$. Nevertheless, it is routine to check
that $\HH$ is a linear intersecting $8$-partite hypergraph.
Suppose that $X$ is a $6$-cover of $\HH$. Then $X$ must include
$v_1$ and $v_2$ since otherwise it cannot cover the lines through
those vertices. The $24$ lines that avoid $v_1$ and $v_2$ induce
a subhypergraph with maximum degree $4$, which thus cannot be covered
by fewer than $6$ vertices. This contradiction shows that $X$ does not exist.
Since $V_0$ is a $7$-cover, we must have $\tau(\HH)=7$.

Some further properties of $\HH$ are discussed in the next section.

\section{Computational results}\label{sec:comp}

In this section we describe a computational proof of the following
result. 

\begin{theorem}\label{t:comp}
\phantom{kick item 1 onto new line} \
\begin{enumerate}
\item For $r\le7$ the only linear intersecting $r$-partite hypergraphs
  to achieve equality in Ryser's conjecture are subhypergraphs of
  $\P'_r$. In particular, there are none for $r=7$.
\item 
No subhypergraph of $\HH$ has
  $\tau=7$ and is isomorphic to a subhypergraph of $\P'_8$.
\item The smallest subhypergraph of $\HH$ with $\tau=7$ has $22$ lines.
  The smallest subhypergraph of $P'_8$ with $\tau=7$ also has $22$ lines.
\end{enumerate}
\end{theorem}

Clearly, by part (1), every $7$-partite linear intersecting hypergraph
satisfies $\tau\le5$. There are a number of non-isomorphic ways to achieve
$\tau=5$, including by the construction in \lref{l:fromMOLS_2extra_sides}.
For $7$-partite intersecting non-linear hypergraphs with $\tau=6$, see
\cite{AP,ABW15}.

Let $H$ be a $7$-partite linear intersecting hypergraph with
$h=|H|$ lines and $\tau\ge6$. By an argument similar to the proof of
\lref{l:maxdegalmdisj} we know that $\Delta(H)\le6$.  We next argue
that $\Delta(H)\ge4$. Let $H$ have $x_i$ vertices of degree
$i=1,\dots,\Delta$. Note that no line of $H$ can include two vertices
of degree $1$, otherwise the remaining vertices on the line would 
provide a $5$-cover. Hence $x_1\le h$. Together with 
\cite[Lem\,2.1]{ABW15} and \cite[Thm\,2.7]{ABW15}, 
we can then deduce that if $\Delta\le 4$ and $x_4\le 7$ then $x_1=h=17$,
$x_3\ge38$ and hence $x_2<0$. It follows that $\Delta\ge4$. Also
if $\Delta=4$ then $x_4>7$ so some side of $H$ has at least $2$ vertices
of degree $4$ on it.

We next describe the computation that established Part 1 of \tref{t:comp}.
By the above comments, we can split the problem for $r=7$ into three subcases 
$\Delta=4$, $\Delta=5$ and $\Delta=6$. We started with a vertex of degree
$\Delta$ on side $V_0$. (In the $\Delta=4$ case, we then added a second
vertex of degree $\Delta$ to $V_0$ in all possible ways up to isomorphism.)
Subsequent lines were added one at a time, ensuring that all pairs of lines 
intersected in a single point and that the assumed maximum degree was not
violated. After each line was added, we tested for isomorphism and kept
only one representative of each isomorphism class. 
For isomorphism checking we converted the hypergraphs into 
vertex-coloured graphs and applied the software \textit{nauty}~\cite{nauty}.
For $\Delta=4,5,6$ the
largest hypergraphs we obtained had $16,25,18$ lines respectively.
All hypergraphs that we built had a $5$-cover, proving the claim that
no linear $7$-partite intersecting hypergraph achieves $\tau=6$.

For $r\le6$, we performed computations as just described, except that
there was no need to split the problem into subcases according to the
maximum degree. Every hypergraph that we encountered could be extended
to $\P'_r$.

For $r\ge8$ the above method is not practical for a complete
enumeration. However, we did a partial enumeration and found a number
of linear intersecting $8$-partite hypergraphs that are maximal (no
lines can be added), have $\tau=7$ and yet are not isomorphic to
$\P'_8$. Most of these have the property that a few lines can be
removed to get something isomorphic to a subhypergraph of $\P'_8$.
However, the hypergraph $\HH$ described in \sref{s:H38} seems to be of a
very different nature, which is why we tested its properties more 
thoroughly.

\begin{table}
\centering
\begin{tabular}{|c|r|}
\hline
$|H|$&Number\\
\hline
22&         833    \\
23&     2168877    \\
24&    58227758    \\
25&   224055209    \\
26&   368614512    \\
27&   401984117    \\
28&   351960321    \\
\hline
\end{tabular}
\quad
\begin{tabular}{|c|r|}
\hline
$|H|$&Number\\
\hline
29&   268297692    \\
30&   183765292    \\
31&   114391098    \\
32&    64949914    \\
33&    33653522    \\
34&    15894680    \\
35&     6828374    \\
\hline
\end{tabular}
\quad
\begin{tabular}{|c|r|}
\hline
$|H|$&Number\\
\hline
36&     2660309    \\
37&      936491    \\
38&      296473    \\
39&       84035    \\
40&       21221    \\
41&        4757    \\
42&         953    \\
\hline
\end{tabular}
\quad
\begin{tabular}{|c|r|}
\hline
$|H|$&Number\\
\hline
43&         179    \\
44&          32    \\
45&           8    \\
46&           3    \\
47&           1    \\
48&           1    \\
49&           1    \\
\hline
\end{tabular}
\caption{\label{T:subPP}Number of isomorphism classes of subhypergraphs
$H$ of $\P'_8$, with $\tau(H)=7$.}
\end{table}

In \Tref{T:subPP} and \Tref{T:subH38}
all 2098796663 isomorphism classes of subhypergraphs
of $\P'_8$ with $\tau=7$, and all 17892655 isomorphism classes of subhypergraphs
of $\HH$ with $\tau=7$ are classified by their size.
The tables were prepared by exhaustive enumeration, using a heuristic
upper bound for the covering number to quickly eliminate most subhypergraphs
with $\tau(H)\le 6$, and employing \textit{nauty} to remove isomorphs.

\begin{table}
\centering
\begin{tabular}{|c|r|}
\hline
$|H|$&Number\\
\hline
22&         5    \\
23&     42310    \\
24&    1550265    \\
25&   5027821    \\
26&   5332373    \\
\hline
\end{tabular}
\quad
\begin{tabular}{|c|r|}
\hline
$|H|$&Number\\
\hline
27&   3376797    \\
28&   1625274    \\
29&   644482    \\
30&    215066    \\
31&    60609    \\
\hline
\end{tabular}
\quad
\begin{tabular}{|c|r|}
\hline
$|H|$&Number\\
\hline
32&     14308    \\
33&      2803    \\
34&      462    \\
35&       67    \\
36&       9    \\
\hline
\end{tabular}
\quad
\begin{tabular}{|c|r|}
\hline
$|H|$&Number\\
\hline
37&         3    \\
38&          1    \\
&              \\
&              \\
&              \\
\hline
\end{tabular}
\caption{\label{T:subH38}Number of isomorphism classes of subhypergraphs
$H$ of $\HH$, with $\tau(H)=7$.}
\end{table}

We end with an example of a subhypergraph of $\P'_8$ that has 22 lines and  
$\tau=7$:
\[\begin{array}{llllll}
03426434&\quad 04505645&\quad 06663521&\quad 11264344&\quad 15636215&\quad 16055456\\
22642443&\quad 24366152&\quad 25550564&\quad 32133654&\quad 34624066&\quad 35345331\\
43331546&\quad 44453313&\quad 46246660&\quad 51313465&\quad 55462606&\quad 56534133\\
60444555&\quad 61651632&\quad 62516326&\quad 63165263
\end{array}\]
and an example of a subhypergraph of $\HH$ that has 22 lines and  
$\tau=7$:
\[\begin{array}{llllll}
00000008&\quad 03615429&\quad 10536249&\quad 20361454&\quad 22222228&\quad 24510635\\
25043169&\quad 26105341&\quad 31402659&\quad 33333338&\quad 36541026&\quad 42016356\\
43562101&\quad 44444448&\quad 46320519&\quad 54306123&\quad 56413204&\quad 60425136\\
62503412&\quad 63140253&\quad 64251309&\quad 65312047
\end{array}\]
Both these hypergraphs have an automorphism group of order $3$, which is the
largest achieved by subhypergraphs with $h=22$ and $\tau=7$
within $\P'_8$ and $\HH$, respectively.

  \let\oldthebibliography=\thebibliography
  \let\endoldthebibliography=\endthebibliography
  \renewenvironment{thebibliography}[1]{%
    \begin{oldthebibliography}{#1}%
      \setlength{\parskip}{0.4ex plus 0.1ex minus 0.1ex}%
      \setlength{\itemsep}{0.4ex plus 0.1ex minus 0.1ex}%
  }%
  {%
    \end{oldthebibliography}%
  }

\end{document}